\documentclass[12pt, a4paper]{article}
\usepackage{}
\usepackage{mathrsfs}
\usepackage{amsfonts}

\usepackage{bbm}
\usepackage{latexsym}
\usepackage{graphicx}
\usepackage{amssymb}

\newtheorem{theorem}{Theorem}[section]
\newtheorem{lemma}[theorem]{Lemma}

\title{{\bf Distance spectral radius of a tree with given diameter\thanks{Supported by NSFC
(Nos. 11271315, 11101057, 11171290, 11201407) and BK2012245.}}}
\author{Guanglong Yu$^{a}$\thanks{ E-mail addresses:
yglong01@163.com (Yu).} ~ Shuguang Guo$^{a}$ ~ Mingqing Zhai$^{b}$  ~
\\ ~ \\
{\footnotesize $^a$Department of Mathematics, Yancheng Teachers
University,  Yancheng, 224002, Jiangsu, P.R.
China}\\
{\footnotesize $^b$Department of Mathematics, Chuzhou University, Chuzhou, 239012, Anhui, P.R. China}}
\date{}

\begin{document}
\maketitle

\begin{abstract}
For a connected graph, the distance spectral radius is the largest
eigenvalue of its distance matrix.
In this paper, of all trees with both given order and fixed diameter, the trees with the
minimal distance spectral radius are completely characterized.


\end{abstract}

\section{Introduction}

\ \ \ \ We use
standard terminology and notation. In a connected graph $G$, the $distance$ between
vertices $v_{i}$ and $v_{j}$, denoted by $d_{G}(v_{i}, v_{j})$ or $d(v_{i}, v_{j})$, is defined to be
the length (i.e. the number of edges) of the shortest path from $v_{i}$
to $v_{j}$. The $distance$ $matrix$ of $G$, denoted by $D(G)=(d_{i,j})_{n\times n}$, is the
$n\times n$ matrix with its $(i, j)$-entry $d_{i,j}$ equal to $d(v_{i}, v_{j})$,
$i, j=1, 2, \ldots, n$. The largest eigenvalue of $D(G)$, denoted by $\varrho(G)$, is called the $distance$ $spectral$
$radius$ of $G$. By the
Perron-Frobenius theorem \cite{OP}, there exists a positive vector corresponding to $\varrho(G)$. We call the unit positive
vector corresponding to $\varrho(G)$ the $distance$ $Perron$ $vector$ of graph $G$.

In \cite{A.T.B}, Balaban et al. proposed the use
of distance spectral radius as a molecular descriptor, while in
\cite{I.G.M}, distance spectral radius was successfully used to infer the extent of
branching and model boiling points of alkane. Therefore, the study
concerning the maximum (minimum) distance spectral radius of a given class of graphs is of great
interest and significance. Recently, the maximum (minimum) distance spectral radius of a given class
of graphs has been studied extensively. The reader is refereed to see \cite{A.I.D, G.I.S}, \cite{DSG}-\cite{B.Z2}.

We now introduce some other notations. For a connected graph $G$, we denote by $N_{G}(v)$ the neighbor
set and $deg(v)$ the degree of vertex $v$ respectively. We denote by $K_{n}$, $P_{n}$ a complete graph, a path of order $n$ respectively. A graph $G$ is called $nontrivial$ if $|V(G)|\geq 2$. Otherwise, $G$ is called trivial. The $diameter$ of a graph is the maximum distance between any pair of vertices. In a graph, if the length of a path $P$ is both equal to the distance between its two end vertices and the diameter of the graph, then
the path $P$ is called a $D$-$path$ of this graph. The $union$ of two simple graphs $H$ and $G$ is the simple graph $G\cup H$ with vertex set $V(G)\cup V(H)$
 and edge set $E(G)\cup E(H)$.

In \cite{X.L.Z}, among all
connected graphs of order $n$ with given diameter, X. Zhang determined the graphs with the minimum distance spectral radius. In \cite{DIF} and \cite{Y.J.Z}, Z. Du et al. and Y. Zhang et al. investigated respectively which trees have the the minimum distance spectral radius among all
trees of order $n$ with given diameter. They only determined the trees with the minimum distance spectral radius among all
trees with even diameter.

Denote by $\mathcal {T}_{1}$ the graph obtained by attaching $n-2k$ $(k\geq2)$ pendant edges to vertex $v_{1}$ of path $P_{2k}=v_{k}v_{k-1}\cdots v_{1}u_{1}u_{2}\cdots u_{k}$ (see Fig. 1.1).

\setlength{\unitlength}{0.6pt}
\begin{center}
\begin{picture}(599,108)
\put(19,46){\circle*{4}}
\put(593,46){\circle*{4}}
\qbezier(19,46)(306,46)(593,46)
\put(336,46){\circle*{4}}
\put(56,46){\circle*{4}}
\put(512,47){\circle*{4}}
\put(217,47){\circle*{4}}
\put(255,46){\circle*{4}}
\put(299,47){\circle*{4}}
\put(553,46){\circle*{4}}
\put(92,46){\circle*{4}}
\put(248,34){$v_{1}$}
\put(212,34){$v_{2}$}
\put(0,45){$v_{k}$}
\put(46,32){$v_{k-1}$}
\put(80,32){$v_{k-2}$}
\put(292,33){$u_{1}$}
\put(329,33){$u_{2}$}
\put(600,45){$u_{k}$}
\put(544,33){$u_{k-1}$}
\put(502,34){$u_{k-2}$}
\put(219,-9){Fig. 1.1. $\mathcal {T}_{1}$}
\put(219,76){\circle*{4}}
\qbezier(255,46)(237,61)(219,76)
\put(238,88){\circle*{4}}
\qbezier(255,46)(246,67)(238,88)
\put(299,76){\circle*{4}}
\qbezier(255,46)(277,61)(299,76)
\put(282,82){\circle*{4}}
\put(269,85){\circle*{4}}
\put(257,87){\circle*{4}}
\put(186,81){$v_{2k+1}$}
\put(220,96){$v_{2k+2}$}
\put(305,74){$v_{n}$}
\end{picture}
\end{center}

For the trees of order $n$ with odd diameter $d=2k-1$, in \cite{Y.J.Z}, the authors conjectured that the tree with the minimum distance spectral radius is uniquely obtained at $\mathcal {T}_{1}$. In this article, we give an affirmative answer for this
conjecture.

\section{Preliminary}

\begin{lemma}{\bf \cite{A.I.D}} \label{le2,1}
Let $w$ be a vertex of a nontrivial connected graph $G$, and for
nonnegative integers $p$ and $q$, let $G(p,q)$ denote the graph
obtained from $G$ by attaching two pendant paths $P=wv_{1}v_{2}\cdots
v_{p}$ and $Q=wu_{1}u_{2}\cdots u_{q}$. If $p\geq q\geq 1$, then
$$\varrho(G(p+1,q-1))> \varrho(G(p,q)).$$
\end{lemma}

Denote by $G_{uv}$ the graph
obtained from graph $G$ by contracting the edge $uv$, that is, deleting the edge $uv$
and identifying the two vertices $u$ and $v$.

\unitlength 1mm \linethickness{0.4pt}
\begin{center}
\begin{picture}(100.67,28.137)
\put(8.33,12.00){\circle{14.00}} \put(34.67,12.00){\circle{14.00}}
\put(15.50,12.00){\circle*{1.33}} \put(27.67,12.00){\circle*{1.33}}
\put(16.00,12.00){\line(1,0){12.33}}
\put(4.33,12.00){\makebox(0,0)[cc]{$G_{1}$}}
\put(39.33,12.33){\makebox(0,0)[cc]{$G_{2}$}}
\put(17.00,10.00){\makebox(0,0)[cc]{$u$}}
\put(26.67,10.00){\makebox(0,0)[cc]{$v$}}
\put(21.00,3.00){\makebox(0,0)[cc]{$G$}}
\put(46.33,12.67){\line(1,0){10.33}}
\put(46.67,11.00){\line(1,0){10.00}}
\put(54.00,14.50){\line(4,-3){3.60}}
\put(57.67,12.00){\line(-4,-3){3.60}}
\put(70.00,12.00){\circle{14.00}} \put(84.33,12.00){\circle{14.00}}
\put(77.00,12.00){\circle*{1.33}}
\put(77.00,12.00){\line(0,1){10.33}}
\put(77.00,22.00){\circle*{1.33}}
\put(66.33,12.00){\makebox(0,0)[cc]{$G_{1}$}}
\put(88.67,12.33){\makebox(0,0)[cc]{$G_{2}$}}
\put(77.00,7.67){\makebox(0,0)[cc]{$u$}}
\put(78.67,23.67){\makebox(0,0)[cc]{$v$}}
\put(76.67,3.00){\makebox(0,0)[cc]{$G^{'}$}}
\put(45.67,-3.00){\makebox(0,0)[cc]{Fig. 2.1. $G$ and $G^{'}$}}
\end{picture}
\end{center}

\begin{lemma} \label{le2,2}{\bf \cite{Y.J.Z}} 
Suppose $uv$ is a cut-edge of connected graph $G$, but $uv$ is not a
pendant edge. Let $u$ denote the vertex obtained from identifying
$u$ and $v$ in $G_{uv}$, and $G^{'}=G_{uv}+uv$. Then
~$\varrho(G)>\varrho(G')$ (see Fig. 2.1).
\end{lemma}

\setlength{\unitlength}{0.5pt}
\begin{center}
\begin{picture}(609,137)
\qbezier(209,91)(209,71)(222,58)\qbezier(222,58)(235,45)(255,45)\qbezier(255,45)(274,45)(287,58)
\qbezier(287,58)(301,71)(301,91)
\qbezier(209,91)(209,110)(222,123)\qbezier(222,123)(235,137)(255,137)\qbezier(255,137)(274,137)(287,123)
\qbezier(287,123)(301,110)(301,91)
\put(19,46){\circle*{5}}
\put(593,46){\circle*{5}}
\qbezier(19,46)(306,46)(593,46)
\put(336,46){\circle*{5}}
\put(56,46){\circle*{5}}
\put(512,47){\circle*{5}}
\put(217,47){\circle*{5}}
\put(255,46){\circle*{5}}
\put(299,47){\circle*{5}}
\put(553,46){\circle*{5}}
\put(92,46){\circle*{5}}
\put(249,104){$G^{'}$}
\put(248,30){$v_{1}$}
\put(212,30){$v_{2}$}
\put(-5,45){$v_{k}$}
\put(36,30){$v_{k-1}$}
\put(80,30){$v_{k-2}$}
\put(292,30){$u_{1}$}
\put(329,30){$u_{2}$}
\put(600,45){$u_{k}$}
\put(544,30){$u_{k-1}$}
\put(492,30){$u_{k-2}$}
\put(200,-9){Fig. 2.2. $G$}
\end{picture}
\end{center}

\begin{lemma}\label{le2,3} 
Let $G$ be a graph obtained by attaching two pendant paths of length $k-1$ and $k$ $(k\geq 2)$ respectively to a vertex of a nontrivial graph $G^{'}$ (see Fig. 2.2). Denote by $X=\{x_{u_{1}}$,  $x_{u_{2}}$, $\ldots$,  $x_{u_{k}}$,  $x_{v_{1}}$,  $x_{v_{2}}$, $\ldots$,  $x_{v_{k}}$, $x_{u}$, $x_{v}$, $\ldots\}^{T}$ the distance Perron vector of $G$, where $x_{w}$ corresponds to vertex $w$. Then we have
$x_{v_{1}}<x_{u_{1}}$ and $x_{u_{i-1}}<x_{v_{i}}<x_{u_{i}}$ for $i=2$, $3$, $\ldots$, $k$.

\end{lemma}

\begin{proof}
We first prove $x_{v_{k}}<x_{u_{k}}$. Otherwise, suppose $x_{v_{k}}\geq x_{u_{k}}$. Note that for $i=1$, $2$, $\ldots$,
$k-1$,
$$\varrho(G)(x_{v_{i}}- x_{u_{i}})-\varrho(G)(x_{v_{i+1}}- x_{u_{i+1}})=2\sum^{k}_{j=i+1}(x_{v_{j}}- x_{u_{j}}).$$
Then $$\varrho(G)(x_{v_{k-1}}- x_{u_{k-1}})-\varrho(G)(x_{v_{k}}- x_{u_{k}})=2(x_{v_{k}}- x_{u_{k}})\geq 0.$$
Noting that $\varrho(G)>0$ and $x_{v_{k}}\geq x_{u_{k}}$, we get $x_{v_{k-1}}\geq x_{u_{k-1}}$.
By induction, we get $x_{v_{i}}\geq x_{u_{i}}$ for $i=1$, $2$, $\ldots$,
$k$. Let $V_{0}=\{u_{1}$, $u_{2}$, $\ldots$, $u_{k}\}$. We find that
$$\varrho(G)(x_{u_{1}}- x_{v_{1}})=\sum_{w\in V(G)\backslash V_{0}} x_{w}-\sum^{k}_{i=1}x_{u_{i}}> 0,$$ which contradicts $x_{v_{1}}\geq x_{u_{1}}$. As a result, we have $x_{v_{k}}<x_{u_{k}}$. As above proof, by induction, we can prove that $x_{v_{i}}<x_{u_{i}}$ for $i=1$, $2$, $\ldots$,
$k$.

Note that $$\displaystyle\varrho(G)(x_{v_{2}}- x_{u_{1}})=2\sum^{k}_{i=1}x_{u_{i}}-2\sum^{k}_{i=2}x_{v_{i}}>0,$$ and note that for $2\leq i\leq k$, $$\displaystyle\varrho(G)(x_{v_{i+1}}- x_{u_{i}})-\varrho(G)(x_{v_{i}}- x_{u_{i-1}})=2\sum^{k}_{j=i}x_{u_{j}}-2\sum^{k}_{j=i+1}x_{v_{j}}>0.$$ By induction, we have
$x_{u_{i-1}}<x_{v_{i}}$ for $i=2$, $3$, $\ldots$,
$k$.

Finally, noting that $x_{v_{1}}<x_{u_{1}}$, we get $$\varrho(G)(x_{u_{2}}- x_{u_{1}})-\varrho(G)(x_{u_{1}}- x_{v_{1}})=2x_{u_{1}}> 0,$$ and then $x_{u_{2}}> x_{u_{1}}$. Note that for $i=2$, $3$, $\ldots$,
$k$, $$\varrho(G)(x_{u_{i+1}}- x_{u_{i}})-\varrho(G)(x_{u_{i}}- x_{u_{i-1}})=2x_{u_{i}}> 0.$$ As above proof, by induction, we can prove that
$x_{u_{i-1}}<x_{u_{i}}$ for $i=2$, $3$, $\ldots$, $k$. This completes the proof. \ \ \ \ \ $\Box$

\end{proof}

\begin{lemma}\label{le2,4} 
Denote by $X=\{x_{u_{1}}$,  $x_{u_{2}}$, $\ldots$,  $x_{u_{k}}$,  $x_{v_{1}}$,  $x_{v_{2}}$, $\ldots$,  $x_{v_{k}}$, $x_{v_{2k+1}}$, $x_{v_{2k+2}}$, $\ldots$, $x_{v_{n}}\}^{T}$ the distance Perron vector of $\mathcal {T}_{1}$ (see Fig. 1.1), where $x_{w}$ corresponds to vertex $w$. If $n\geq 2k+1$, then we have
$x_{v_{2k+1}}=x_{v_{2k+2}}=\cdots =x_{v_{n}}$ and $\displaystyle x_{v_{1}}+\lceil\frac{n-2k}{2}\rceil x_{v_{n}}>x_{u_{k}}.$
\end{lemma}

\begin{proof}
By symmetry, we have $x_{v_{2k+1}}=x_{v_{2k+2}}=\cdots =x_{v_{n}}$. We claim that $x_{v_{k}}\geq x_{v_{2k+1}}$. Otherwise, suppose $x_{v_{k}}< x_{v_{2k+1}}$. Then
$$\varrho(G)(x_{v_{1}}- x_{v_{k-1}})-\varrho(G)(x_{v_{2k+1}}- x_{v_{k}})=2(x_{v_{2k+1}}- x_{v_{k}})> 0.$$ By Lemma \ref{le2,3}, we know that it is impossible. Then our claim holds. In particular, if $k>2$, we can prove $x_{v_{k}}> x_{v_{2k+1}}$ similarly.

Noting that $x_{v_{i}}<x_{u_{i}}$ for $i=1$, $2$, $\ldots$, $k$ by Lemma \ref{le2,3}, we have $$\varrho(G)x_{v_{1}}=\sum^{k}_{i=1}(i-1)x_{v_{i}}+\sum^{n}_{i=2k+1}x_{v_{i}}+\sum^{k}_{i=1}ix_{u_{i}}\ \ \ \ \ \ \ \ \ \ \ \ \ \ \ \ \ \ $$
$$=\sum^{k}_{i=1}((i-1)x_{v_{i}}+ix_{u_{i}})+\sum^{n}_{i=2k+1}x_{v_{i}}\ \ \ \ \ \ \ \ \ $$
$$>\sum^{k}_{i=1}(2i-1)x_{v_{i}}+(n-2k)x_{v_{n}}\ \ \ \ \ \ \ \ \ \ \ \ \ \ \ $$
$$\ \ \ =\sum^{k}_{i=1}(k+i-1-(k-i))x_{v_{i}}+(n-2k)x_{v_{n}}.\ \  $$

Noting that $x_{v_{k}}\geq x_{v_{2k+1}}$ and $x_{u_{k}}\geq x_{v_{k}}$, we have
$$\varrho(G)\sum^{2k+\lceil\frac{n-2k}{2}\rceil}_{i=2k+1}x_{v_{i}}=
\lceil\frac{n-2k}{2}\rceil(\sum^{k}_{i=1}ix_{v_{i}}+\sum^{k}_{i=1}(i+1)x_{u_{i}}+2(n-2k-1)x_{v_{n}})\ \ \ \ \ \ \ \ $$
$$=\lceil\frac{n-2k}{2}\rceil kx_{v_{k}}+\lceil\frac{n-2k}{2}\rceil(k+1)x_{u_{k}}\ \ \ \ \ \ \ \ $$
$$\ \ \ \ \ \ \ \ \ \ \ \ \ \ \ \ \ \ \ \ \ +\lceil\frac{n-2k}{2}\rceil(\sum^{k-1}_{i=1}ix_{v_{i}}+\sum^{k-1}_{i=1}(i+1)x_{u_{i}}+2(n-2k-1)x_{v_{n}})$$
$$\ \ \ >(n-2k)kx_{v_{k}}+\lceil\frac{n-2k}{2}\rceil x_{u_{k}}\ \ \ \ \ \ \ \ \ \ \ \ \ \ \ \ \ \ \ \ \ $$ $$\ \ \ \ \ \ \ \ \ \ \ \ \ \ \ \ \ \ \ \ \ \ +\sum^{k-1}_{i=1}ix_{v_{i}}+\lceil\frac{n-2k}{2}\rceil(\sum^{k-1}_{i=1}(i+1)x_{u_{i}}+2(n-2k-1)x_{v_{n}}).$$

Noting that $x_{v_{k}}> x_{v_{2k+1}}$, $x_{v_{i-1}}<x_{v_{i}}$ for $i=2$, $3$, $\ldots$, $k$, and $\displaystyle \sum^{k-1}_{i=1}ix_{v_{i}}=\sum^{k-1}_{i=1}\sum^{k-i}_{j=1}x_{v_{k-j}}$, we have
$$\varrho(G)(x_{v_{1}})+\varrho(G)\sum^{2k+\lceil\frac{n-2k}{2}\rceil}_{i=2k+1}x_{v_{i}}\ \ \ \ \ \ \ \ \ \ \ \ \ \ \ \ \ \ \ \ \ \ \ \ \ \ \ \ \ \ \ \ \ \ \ \ \ \ \ \ \ \ \ \ \ \ \ \ \ \ \ \ \ \ \ \ \ \ \ \ \ \ \ \ \ \ \ \ \ \ $$
$$>\sum^{k}_{i=1}((k+i-1-(k-i))x_{v_{i}}+\sum^{k-i}_{j=1}x_{v_{k-j}})+(n-2k)(k+1)x_{v_{k}}+\lceil\frac{n-2k}{2}\rceil x_{u_{k}}\ \ \ \ $$
$$ +\lceil\frac{n-2k}{2}\rceil(\sum^{k-1}_{i=1}(i+1)x_{u_{i}}+2(n-2k-1)x_{v_{n}})\ \ \ \ \ \ \ \ \ \ \ \ \ \ \ \ \ \ \ \ \ \ \ \ \ \ \ \ \ \ \ \ \ \ \ \ \ \ \ \ \ \ $$
$$>\sum^{k}_{i=1}(k+i-1)x_{v_{i}}+(n-2k)(k+1)x_{v_{k}}+\lceil\frac{n-2k}{2}\rceil x_{u_{k}}\ \ \ \ \ \ \ \ \ \ \ \ \ \ \ \ \ \ \ \ \ \ \ \ \ \ \ \ \ \ \ \ \ \ $$
$$ +\lceil\frac{n-2k}{2}\rceil(\sum^{k-1}_{i=1}(i+1)x_{u_{i}}+2(n-2k-1)x_{v_{n}}).\ \ \ \ \ \ \ \ \ \ \ \ \ \ \ \ \ \ \ \ \ \ \ \ \ \ \ \ \ \ \ \ \ \ \ \ \ \ \ \ \ \ $$

Noting that $\displaystyle \varrho(G)x_{u_{k}}=\sum^{k}_{i=1}(k+i-1)x_{v_{i}}+(n-2k)(k+1)x_{v_{n}}+\sum^{k}_{i=1}(k-i)x_{u_{i}}$ and $x_{u_{i-1}}<x_{u_{i}}$ for $i=2$, $3$, $\ldots$, $k$,
we have $$\varrho(G)(x_{v_{1}})+\varrho(G)\sum^{2k+\lceil\frac{n-2k}{2}\rceil}_{i=2k+1}x_{v_{i}}-\varrho(G)x_{u_{k}}\ \ \ \ \ \ \ \ \ \ \ \ \ \ \ \ \ \ \ \ \ \ \ \ \ \ \ \ \ \ \ \ \ \ \ \ \ \ \ \ \ \ \ \ \ \ \ \ \ \ \ \ \ \ \ \ \ \ \ \ \ \ \ \ \ $$
$$\ \ \ \ \ \ \ >\lceil\frac{n-2k}{2}\rceil x_{u_{k}}+\lceil\frac{n-2k}{2}\rceil(\sum^{k-1}_{i=1}(i+1)x_{u_{i}}+2(n-2k-1)x_{v_{n}})-\sum^{k}_{i=1}(k-i)x_{u_{i}}\ \ \ \ \ \ \ \ \ \ \ \ \ \ \ \ \ \ \ \ \ \ \ $$
$$>\lceil\frac{n-2k}{2}\rceil x_{u_{k}}+\sum^{\lfloor\frac{k-1}{2}\rfloor}_{i=1}((k-2i+1)x_{u_{k-i}}-(k-2i-1)x_{u_{i}}))>0.\ \ \ \ \ \ \ \ \ \ \ \ \ \ \ \ \ \ \ \ \ $$ Then $$x_{v_{1}}+\sum^{2k+\lceil\frac{n-2k}{2}\rceil}_{i=2k+1}x_{v_{i}}>x_{u_{k}}$$ follows.
This completes the proof. \ \ \ \ \ $\Box$

\end{proof}

\section{Extremal graphs}

\setlength{\unitlength}{0.6pt}
\begin{center}
\begin{picture}(522,125)
\qbezier(81,59)(10,98)(2,60) \qbezier(2,60)(0,4)(81,59)
\qbezier(81,59)(159,99)(161,59) \qbezier(81,59)(159,17)(161,59)
\qbezier(81,59)(117,127)(72,126) \put(75,44){$v_{0}$}
\put(16,56){$G_{1}$} \put(133,54){$G_{2}$} \put(70,104){$G_{3}$}
\put(77,17){$G$} \qbezier(72,126)(32,116)(81,59)
\qbezier(299,60)(306,103)(380,59) \qbezier(299,60)(304,12)(380,59)
\qbezier(380,59)(461,115)(472,61) \qbezier(380,59)(465,7)(472,61)
\qbezier(380,59)(416,59)(452,59) \qbezier(452,59)(442,139)(484,125)
\qbezier(452,59)(522,102)(484,125) \put(374,45){$v_{0}$}
\put(445,47){$v_{a}$} \put(412,64){$P_{a,0}$} \put(311,56){$G_{1}$}
\put(472,40){$G_{2}$} \put(468,105){$G_{3}$} \put(380,17){$H$}
\qbezier(194,67)(221,67)(249,67) \qbezier(194,55)(221,55)(249,55)
\qbezier(239,74)(249,68)(260,61) \qbezier(260,61)(250,55)(240,48)
\put(453,59){\circle*{4}} \put(380,59){\circle*{4}}
\put(80,59){\circle*{4}} \put(170,-11){Fig. 3.1. $G, H$}
\end{picture}
\end{center}

\begin{lemma}{\bf \cite{YJZS}} \label{le3,1} 
Suppose graph $G=\bigcup^{3}_{i=1} G_{i}$ satisfies
that $G_{i}\cap G_{j}=v_{0}$ for $1\leq i, j\leq 3$, $i\neq j$,
and that $|V(G_{i})|\geq 2$ for $i=1, 2, 3$ (see\ Fig.\ 3.1).
$X=(x_{0}$, $x_{1}$, $\cdots$, $x_{n-1})^{T}$ is the Perron
eigenvector corresponding to $\varrho(G)$, in which $x_{i}$
corresponds to $v_{i}$. Let $\displaystyle S_{1}=\sum_{v_{i}\in
V(G_{1})}x_{i}, \displaystyle S_{2}=\sum_{v_{i}\in V(G_{2})}x_{i}$,
and for a vertex $v_{a}\in V(G_{2}), v_{a}\neq v_{0}$, let graph
$\displaystyle H=G-\sum_{v_{i}\in
N_{G_{3}}(v_{0})}v_{0}v_{i}+\sum_{v_{i}\in
N_{G_{3}}(v_{0})}v_{a}v_{i}.$ If $S_{1}\geq S_{2}$, then
$\varrho(H)> \varrho(G)$.

\end{lemma}

\begin{lemma}\label{le3,2} 
Of all the trees with $n$ vertices and diameter $d=2k-1$ $(k\geq 2)$, the
minimal distance spectral radius is obtained uniquely at $\mathcal {T}_{1}$.

\end{lemma}

\begin{proof}
This lemma is trivial for $n=2k$.

If $n=2k+1$, suppose that $P=v_{k}v_{k-1}\cdots v_{1}u_{1}u_{2}\cdots u_{k}$ is a $D$-path in a tree $\mathcal {T}$ with diameter $d=2k-1$, and suppose the vertex $v_{n}$ is attached to vertex $v_{i}$ $(1\leq i\leq k-1)$. Denote by $X=\{x_{u_{1}}$,  $x_{u_{2}}$, $\ldots$,  $x_{u_{k}}$,  $x_{v_{1}}$,  $x_{v_{2}}$, $\ldots$,  $x_{v_{k}}$, $x_{v_{2k+1}}\}^{T}$ the distance Perron vector of $\mathcal {T}_{1}$, where $x_{w}$ corresponds to vertex $w$. Let $\displaystyle S_{1}=\sum^{k}_{i=1}x_{u_{i}}$, $\displaystyle S_{2}=\sum^{k}_{i=1}x_{v_{i}}$. By Lemma \ref{le2,3}, we see that $S_{1}> S_{2}$. By Lemma \ref{le3,1}, we get that $\varrho(\mathcal {T})\geq \varrho(\mathcal {T}_{1})$, with equality if and only if $\mathcal {T}\cong \mathcal {T}_{1}$. Hence this lemma holds for $n=2k+1$.

Next we prove this lemma holds for the case that $n\geq 2k+2$. To prove this, we let $\mathscr{T}$ be a tree with $n\geq 2k+2$ vertices and diameter $d=2k-1$, and suppose that $\mathcal {P}=v_{k}v_{k-1}\cdots v_{1}u_{1}u_{2}\cdots u_{k}$ is a $D$-path of $\mathscr{T}$. We consider three cases as follows.

{\bf Case 1} In $\mathscr{T}$, for $1\leq i\leq k-1$, $deg(u_{i})=2$, but there exists at least two vertices in $\{v_{1}$, $v_{2}$, $\ldots$, $v_{k-1}\}$ with degree more than $2$.

\setlength{\unitlength}{0.5pt}
\begin{center}
\begin{picture}(638,110)
\put(20,43){\circle*{4}}
\put(634,42){\circle*{4}}
\put(216,81){\circle*{4}}
\put(260,42){\circle*{4}}
\put(588,42){\circle*{4}}
\put(132,42){\circle*{4}}
\put(58,43){\circle*{4}}
\put(198,42){\circle*{4}}
\put(324,42){\circle*{4}}
\put(379,42){\circle*{4}}
\put(420,43){\circle*{4}}
\qbezier(20,43)(39,43)(58,43)
\put(147,79){\circle*{4}}
\put(110,42){\circle*{4}}
\put(94,42){\circle*{4}}
\put(80,42){\circle*{4}}
\put(179,42){\circle*{4}}
\put(165,42){\circle*{4}}
\put(150,42){\circle*{4}}
\put(112,80){\circle*{4}}
\qbezier(132,42)(122,61)(112,80)
\put(158,78){\circle*{4}}
\qbezier(132,42)(145,60)(158,78)
\put(136,80){\circle*{4}}
\put(125,81){\circle*{4}}
\put(179,81){\circle*{4}}
\qbezier(198,42)(188,62)(179,81)
\put(227,81){\circle*{4}}
\qbezier(198,42)(212,62)(227,81)
\put(204,82){\circle*{4}}
\put(192,81){\circle*{4}}
\put(278,42){\circle*{4}}
\put(290,42){\circle*{4}}
\put(303,42){\circle*{4}}
\put(239,43){\circle*{4}}
\put(228,43){\circle*{4}}
\put(218,43){\circle*{4}}
\put(242,85){\circle*{4}}
\qbezier(260,42)(251,64)(242,85)
\put(286,84){\circle*{4}}
\qbezier(260,42)(273,63)(286,84)
\put(350,42){\circle*{4}}
\put(362,42){\circle*{4}}
\put(253,84){\circle*{4}}
\put(338,42){\circle*{4}}
\qbezier(379,42)(506,42)(634,42)
\put(373,27){$v_{1}$}
\put(412,27){$u_{1}$}
\put(577,27){$u_{k-1}$}
\put(639,44){$u_{k}$}
\put(249,27){$v_{a_{1}}$}
\put(186,27){$v_{a_{2}}$}
\put(120,27){$v_{a_{s}}$}
\put(252,101){$\mathcal {S}_{a_{1}}$}
\put(274,84){\circle*{4}}
\put(264,84){\circle*{4}}
\put(190,99){$\mathcal {S}_{a_{2}}$}
\put(122,97){$\mathcal {S}_{a_{s}}$}
\put(-3,48){$v_{k}$}
\put(44,27){$v_{k-1}$}
\put(270,-15){Fig. 3.2. $\mathscr{T}^{'}$}
\end{picture}
\end{center}

Suppose that $T_{a_{j}}$ for $j=1$, $2$, $\ldots$, $s$ ($a_{j}<k$) are the nontrivial connected components containing $v_{a_{j}}$ in $\mathscr{T}-E(\mathcal {P})$. It is easy to see that each $T_{a_{j}}$ is a tree. Let $\mathscr{T}^{'}$ be the graph obtained from $\mathscr{T}$ by transforming each $T_{a_{j}}$ into a star
$\mathcal {S}_{a_{j}}$ with center $v_{a_{j}}$ (see Fig. 3.2). By Lemmas \ref{le2,1}, \ref{le2,2}, we get that $\varrho(\mathscr{T})\geq\varrho(\mathscr{T}^{'})$ with equality if and only if $\mathscr{T}\cong \mathscr{T}^{'}$. Denote by $X=\{x_{u_{1}}$,  $x_{u_{2}}$, $\ldots$,  $x_{u_{k}}$,  $x_{v_{1}}$,  $x_{v_{2}}$, $\ldots$,  $x_{v_{k}}$, $x_{v_{2k+1}}$, $x_{v_{2k+2}}$, $\ldots$, $x_{v_{n}}\}^{T}$ the distance Perron vector of $\mathcal {T}_{1}$, where $x_{w}$ corresponds to vertex $w$. From $\mathcal {T}_{1}$ to $\mathscr{T}^{'}$, we see that

(1) for each $\mathcal {S}_{a_{j}}$, the distance between each vertex $u_{i}$ for $1\leq i\leq k$ and each vertex in $V(\mathcal {S}_{a_{j}})\backslash \{v_{a_{j}}\}$ is increased by $a_{j}-1$;

(2) for each $\mathcal {S}_{a_{j}}$, the distance between each vertex $v_{i}$ for $1\leq i\leq k$ and each vertex in $V(\mathcal {S}_{a_{j}})\backslash \{v_{a_{j}}\}$ is decreased by at most $a_{j}-1$;

(3) for each pair of $\mathcal {S}_{a_{i}}$, $\mathcal {S}_{a_{j}}$ $(i\neq j)$, the distance between each vertex in $V(\mathcal {S}_{a_{i}})\backslash \{v_{a_{i}}\}$ and each vertex in $V(\mathcal {S}_{a_{j}})\backslash \{v_{a_{j}}\}$ is increased by $|a_{i}-a_{j}|$;

(4) the distance between each pair of vertices of $D$-path $\mathcal {P}$ is unchanged, and the distance between each pair of vertices in $V(\mathcal {S}_{a_{j}})\backslash \{v_{a_{j}}\}$ is unchanged.

Let $\displaystyle S_{1}=\sum^{k}_{i=1}x_{u_{i}}$, $\displaystyle S_{2}=\sum^{k}_{i=1}x_{v_{i}}$, and let $\displaystyle \mathcal {C}_{a_{i}}=\sum_{w\in V(\mathcal {S}_{a_{i}})\backslash \{v_{a_{i}}\}}x_{w}$ for $i=1$, $2$, $\ldots$, $s$. By Lemma \ref{le2,3}, we know that $S_{1}> S_{2}$.
Then $$X^{T}D(\mathscr{T}^{'})X-X^{T}D(\mathcal {T}_{1})X\geq 2(S_{1}- S_{2})\sum^{s}_{i=1} (a_{i}-1)\mathcal {C}_{a_{i}}+2\sum_{i\neq j} |a_{i}-a_{j}|\mathcal {C}_{a_{i}}\mathcal {C}_{a_{j}}> 0,$$ which means that $\varrho(\mathscr{T}^{'})> \varrho(\mathcal {T}_{1})$. Noting that $\varrho(\mathscr{T})\geq \varrho(\mathscr{T}^{'})$, we have $\varrho(\mathscr{T})>\varrho(\mathcal {T}_{1})$.

{\bf Case 2} In $\mathscr{T}$, for $1\leq i\leq k-1$, $deg(v_{i})=2$, but there exists at least two vertices in $\{u_{1}$, $u_{2}$, $\ldots$, $u_{k-1}\}$ with degree more than $2$. Similar to Case 1, we can prove that $\varrho(\mathscr{T})>\varrho(\mathcal {T}_{1})$.

{\bf Case 3} There exist both vertex $v_{i}$ ($1\leq i\leq k-1$) and vertex $u_{i}$ ($1\leq j\leq k-1$) such that $deg(v_{i})\geq 3$, $deg(u_{j})\geq 3$.

\setlength{\unitlength}{0.5pt}
\begin{center}
\begin{picture}(616,114)
\put(22,47){\circle*{4}}
\put(610,47){\circle*{4}}
\put(218,85){\circle*{4}}
\put(262,46){\circle*{4}}
\put(563,47){\circle*{4}}
\put(134,46){\circle*{4}}
\put(60,47){\circle*{4}}
\put(200,46){\circle*{4}}
\put(326,46){\circle*{4}}
\put(381,46){\circle*{4}}
\put(423,46){\circle*{4}}
\qbezier(22,47)(41,47)(60,47)
\put(149,83){\circle*{4}}
\put(112,46){\circle*{4}}
\put(96,46){\circle*{4}}
\put(82,46){\circle*{4}}
\put(181,46){\circle*{4}}
\put(167,46){\circle*{4}}
\put(152,46){\circle*{4}}
\put(114,84){\circle*{4}}
\qbezier(134,46)(124,65)(114,84)
\put(160,82){\circle*{4}}
\qbezier(134,46)(147,64)(160,82)
\put(138,84){\circle*{4}}
\put(127,85){\circle*{4}}
\put(181,85){\circle*{4}}
\qbezier(200,46)(190,66)(181,85)
\put(229,85){\circle*{4}}
\qbezier(200,46)(214,66)(229,85)
\put(206,86){\circle*{4}}
\put(194,85){\circle*{4}}
\put(280,46){\circle*{4}}
\put(292,46){\circle*{4}}
\put(305,46){\circle*{4}}
\put(241,47){\circle*{4}}
\put(230,47){\circle*{4}}
\put(220,47){\circle*{4}}
\put(244,89){\circle*{4}}
\qbezier(262,46)(253,68)(244,89)
\put(288,88){\circle*{4}}
\qbezier(262,46)(275,67)(288,88)
\put(352,46){\circle*{4}}
\put(364,46){\circle*{4}}
\put(255,88){\circle*{4}}
\put(340,46){\circle*{4}}
\put(375,27){$v_{1}$}
\put(414,27){$u_{1}$}
\put(556,27){$u_{k-1}$}
\put(617,48){$u_{k}$}
\put(251,27){$v_{a_{1}}$}
\put(188,27){$v_{a_{2}}$}
\put(122,27){$v_{a_{r}}$}
\put(254,105){$\mathcal {S}_{a_{1}}$}
\put(276,88){\circle*{4}}
\put(266,88){\circle*{4}}
\put(192,103){$\mathcal {S}_{a_{2}}$}
\put(124,101){$\mathcal {S}_{a_{r}}$}
\put(-1,52){$v_{k}$}
\put(46,27){$v_{k-1}$}
\put(257,-11){Fig. 3.3. $\mathscr{T}^{'}$}
\qbezier(381,46)(402,46)(423,46)
\qbezier(563,47)(586,47)(610,47)
\put(535,85){\circle*{4}}
\put(469,46){\circle*{4}}
\put(456,46){\circle*{4}}
\put(445,46){\circle*{4}}
\put(434,46){\circle*{4}}
\put(438,82){\circle*{4}}
\qbezier(469,46)(453,64)(438,82)
\put(480,83){\circle*{4}}
\qbezier(469,46)(474,65)(480,83)
\put(515,46){\circle*{4}}
\put(503,47){\circle*{4}}
\put(493,46){\circle*{4}}
\put(483,46){\circle*{4}}
\put(499,84){\circle*{4}}
\qbezier(515,46)(507,65)(499,84)
\put(545,84){\circle*{4}}
\qbezier(515,46)(530,65)(545,84)
\put(550,47){\circle*{4}}
\put(539,47){\circle*{4}}
\put(529,47){\circle*{4}}
\put(523,84){\circle*{4}}
\put(512,83){\circle*{4}}
\put(471,82){\circle*{4}}
\put(458,82){\circle*{4}}
\put(447,82){\circle*{4}}
\put(457,27){$u_{b_{1}}$}
\put(502,27){$u_{b_{2}}$}
\put(446,98){$\mathcal {S}_{b_{1}}$}
\put(510,100){$\mathcal {S}_{b_{t}}$}
\end{picture}
\end{center}

Suppose that in $\mathscr{T}-E(\mathcal {P})$, $T_{a_{j}}$ for $j=1$, $2$, $\ldots$, $r$ ($a_{j}<k$) are the nontrivial connected components containing $v_{a_{j}}$, and $T_{b_{j}}$ for $j=1$, $2$, $\ldots$, $t$ ($b_{j}<k$) are the nontrivial connected components containing $u_{b_{j}}$. Clearly, each $T_{a_{j}}$ and each $T_{b_{j}}$ is a tree. We suppose that
$\sum^{r}_{j=1}(|V(T_{a_{j}})|-1)\geq \sum^{t}_{j=1}(|V(T_{b_{j}})|-1)$. Let $\mathscr{T}^{'}$ be the graph obtained from $\mathscr{T}$ by transforming each $T_{b_{j}}$ into a star $\mathcal {S}_{b_{j}}$ with center $u_{b_{j}}$, and transforming each $T_{a_{j}}$ into a star $\mathcal {S}_{a_{j}}$ with center $v_{a_{j}}$ respectively (see Fig. 3.3). By Lemmas \ref{le2,1}, \ref{le2,2}, we get that $\varrho(\mathscr{T})\geq\varrho(\mathscr{T}^{'})$ with equality if and only if $\mathscr{T}\cong \mathscr{T}^{'}$. Denote by $X=\{x_{u_{1}}$,  $x_{u_{2}}$, $\ldots$,  $x_{u_{k}}$,  $x_{v_{1}}$,  $x_{v_{2}}$, $\ldots$,  $x_{v_{k}}$, $x_{v_{2k+1}}$, $x_{v_{2k+2}}$, $\ldots$, $x_{v_{n}}\}^{T}$ the distance Perron vector of $\mathcal {T}_{1}$, where $x_{w}$ corresponds to vertex $w$. From $\mathcal {T}_{1}$ to $\mathscr{T}^{'}$, we see that

(1) for each $\mathcal {S}_{a_{j}}$, the distance between each vertex $u_{i}$ for $1\leq i\leq k$ and each vertex in $V(\mathcal {S}_{a_{j}})\backslash \{v_{a_{j}}\}$ is increased by $a_{j}-1$;

(2) for each $\mathcal {S}_{a_{j}}$, the distance between each vertex $v_{i}$ for $1\leq i\leq k$ and each vertex in $V(\mathcal {S}_{a_{j}})\backslash \{v_{a_{j}}\}$ is decreased by at most $a_{j}-1$;

(3) for each $\mathcal {S}_{b_{j}}$, the distance between each vertex $v_{i}$ for $1\leq i\leq k$ and each vertex in $V(\mathcal {S}_{b_{j}})\backslash \{u_{b_{j}}\}$ is increased by $b_{j}$;

(4) for each $\mathcal {S}_{b_{j}}$, the distance between each vertex $u_{i}$ for $1\leq i\leq k$ and each vertex in $V(\mathcal {S}_{b_{j}})\backslash \{u_{b_{j}}\}$ is decreased by at most $b_{j}$;

(5) for each pair of $\mathcal {S}_{a_{i}}$, $\mathcal {S}_{a_{j}}$ $(i\neq j)$, the distance between each vertex in $V(\mathcal {S}_{a_{i}})\backslash \{v_{a_{i}}\}$ and each vertex in $V(\mathcal {S}_{a_{j}})\backslash \{v_{a_{j}}\}$ is increased by $|a_{i}-a_{j}|$;

(6) for each pair of $\mathcal {S}_{b_{i}}$, $\mathcal {S}_{b_{j}}$ $(i\neq j)$, the distance between each vertex in $V(\mathcal {S}_{b_{i}})\backslash \{u_{b_{i}}\}$ and each vertex in $V(\mathcal {S}_{b_{j}})\backslash \{u_{b_{j}}\}$ is increased by $|b_{i}-b_{j}|$;

(6) for each pair of $\mathcal {S}_{a_{i}}$, $\mathcal {S}_{b_{j}}$, the distance between each vertex in $V(\mathcal {S}_{a_{i}})\backslash \{v_{a_{i}}\}$ and each vertex in $V(\mathcal {S}_{b_{j}})\backslash \{u_{b_{j}}\}$ is increased by $|a_{i}+b_{j}-1|$;

(7) the distance between each pair of vertices of the $D$-path $\mathcal {P}$ is unchanged, the distance between each pair of vertices in $V(\mathcal {S}_{a_{j}})\backslash \{v_{a_{j}}\}$ and the distance between each pair of vertices in $V(\mathcal {S}_{b_{j}})\backslash \{u_{b_{j}}\}$ are unchanged as well.

Let $\displaystyle S_{1}=\sum^{k}_{i=1}x_{u_{i}}$, $\displaystyle S_{2}=\sum^{k}_{i=1}x_{v_{i}}$, and let $\displaystyle \mathcal {C}_{a_{i}}=\sum_{w\in V(\mathcal {S}_{a_{i}})\backslash \{v_{a_{i}}\}}x_{w}$ for $i=1$, $2$, $\ldots$, $r$, and let $\displaystyle \mathcal {C}_{b_{i}}=\sum_{w\in V(\mathcal {S}_{b_{i}})\backslash \{u_{b_{i}}\}}x_{w}$ for $i=1$, $2$, $\ldots$, $t$. By Lemmas \ref{le2,3}, \ref{le2,4}, we know that $S_{1}> S_{2}$ and $\displaystyle  S_{2}+\sum^{r}_{i=1}\mathcal {C}_{a_{i}}> S_{1}$.
Then $$X^{T}D(\mathscr{T}^{'})X-X^{T}D(\mathcal {T}_{1})X\ \ \ \ \ \ \ \ \ \ \ \ \ \ \ \ \ \ \ \ \ \ \ \ \ \ \ \ \ \ \ \ \ \ \ \ \ \ \ \ \ \ \ \ \ \ \ \ \ $$
$$\geq 2(S_{2}+\sum^{r}_{i=1}\mathcal {C}_{a_{i}}- S_{1})\sum^{s}_{i=1} b_{i}\mathcal {C}_{b_{i}}+2(S_{1}- S_{2})\sum^{s}_{i=1} (a_{i}-1)\mathcal {C}_{a_{i}}$$$$+2\sum_{i\neq j} |b_{i}-b_{j}|\mathcal {C}_{b_{i}}\mathcal {C}_{b_{j}}+2\sum_{i\neq j} |a_{i}-a_{j}|\mathcal {C}_{a_{i}}\mathcal {C}_{a_{j}}> 0.   \ \ \ \ \ \ \ \ $$
 Hence for this case, we have $\varrho(\mathscr{T})>\varrho(\mathcal {T}_{1})$.

From the above three cases, we see that for a tree $\mathscr{T}$ with $n\geq 2k+2$ vertices and diameter $d=2k-1$, $\varrho(\mathscr{T})\geq\varrho(\mathcal {T}_{1})$ with equality if and only if $\mathscr{T}\cong \mathcal {T}_{1}$.

To sum up, we get that for a tree $T$ with $n$ vertices and diameter $d=2k-1$ $(k\geq 2)$, $\varrho(T)\geq\varrho(\mathcal {T}_{1})$ with equality if and only if $T\cong \mathcal {T}_{1}$.
This completes the proof. \ \ \ \ \ $\Box$

\end{proof}

\setlength{\unitlength}{0.6pt}
\begin{center}
\begin{picture}(563,108)
\put(19,46){\circle*{4}}
\put(558,46){\circle*{4}}
\qbezier(19,46)(288,46)(558,46)
\put(336,46){\circle*{4}}
\put(56,46){\circle*{4}}
\put(512,47){\circle*{4}}
\put(217,47){\circle*{4}}
\put(255,46){\circle*{4}}
\put(299,47){\circle*{4}}
\put(92,46){\circle*{4}}
\put(248,34){$v_{0}$}
\put(212,34){$v_{1}$}
\put(-1,45){$v_{k}$}
\put(42,32){$v_{k-1}$}
\put(80,32){$v_{k-2}$}
\put(292,33){$u_{1}$}
\put(329,33){$u_{2}$}
\put(564,46){$u_{k}$}
\put(505,34){$u_{k-1}$}
\put(219,-9){Fig. 3.4. $\mathcal {T}_{2}$}
\put(219,76){\circle*{4}}
\qbezier(255,46)(237,61)(219,76)
\put(238,88){\circle*{4}}
\qbezier(255,46)(246,67)(238,88)
\put(299,76){\circle*{4}}
\qbezier(255,46)(277,61)(299,76)
\put(282,82){\circle*{4}}
\put(269,85){\circle*{4}}
\put(257,87){\circle*{4}}
\put(190,82){$v_{2k+1}$}
\put(224,96){$v_{2k+2}$}
\put(305,74){$v_{n-1}$}
\put(176,46){\circle*{4}}
\put(170,33){$v_{2}$}
\end{picture}
\end{center}

Denote by $\mathcal {T}_{2}$ the graph obtained by attaching $n-2k-1$ pendant edges to vertex $v_{0}$ of path $P_{2k+1}=v_{k}v_{k-1}\cdots v_{1}v_{0}u_{1}u_{2}\cdots u_{k}$ (see Fig. 3.4).

Note that by symmetry, for any eigenvector $X=\{x_{u_{1}}$,  $x_{u_{2}}$, $\ldots$,  $x_{u_{k}}$,  $x_{v_{0}}$,  $x_{v_{1}}$,  $x_{v_{2}}$, $\ldots$,  $x_{v_{k}}$, $\ldots$, $x_{v_{2k+1}}$, $\ldots$, $x_{v_{n-1}}\}^{T}$ corresponding to $\varrho(\mathcal {T}_{2})$, there must be $x_{u_{i}}=x_{v_{i}}$ for $1\leq i\leq k$ and $x_{i_{i}}=x_{v_{j}}$ for $2k+1\leq i<j\leq n-1$.
Similar to the proof of Lemma \ref{le3,2}, we can prove the following lemma which has been shown in \cite{DIF} and \cite{Y.J.Z}.

\begin{lemma}{\bf \cite{DIF}, \cite{Y.J.Z}}\label{le3,3} 
Of all the trees with $n$ vertices and diameter $d=2k$ $(k\geq 1)$, the
minimal distance spectral radius is obtained uniquely at $\mathcal {T}_{2}$.

\end{lemma}

Combining with Lemmas \ref{le3,2}, \ref{le3,3}, we get the following result.

\begin{theorem}\label{th3,2} 
(i) Of all the trees with $n$ vertices and diameter $d=2k-1$ $(k\geq 2)$, the
minimal distance spectral radius is obtained uniquely at $\mathcal {T}_{1}$;

(ii) Of all the trees with $n$ vertices and diameter $d=2k$ $(k\geq 1)$, the
minimal distance spectral radius is obtained uniquely at $\mathcal {T}_{2}$.

\end{theorem}

\noindent{\bf Acknowledgment}

The authors offers many thanks to the referees for their kind comments and helpful
suggestions.

\small {

}

\end{document}